\documentclass[12pt,reqno]{amsproc}
\usepackage{}
\usepackage{mathrsfs, amsfonts, amsthm, amsmath, amssymb}

 \theoremstyle{plain}

\theoremstyle{plain}
\newtheorem{theorem}{Theorem}[section]

\newtheorem{definition}[theorem]{Definition}

\newtheorem{lemma}[theorem]{Lemma}

\newtheorem{remark}[theorem]{Remark}

\makeatletter
\newcommand{\LeftEqNo}{\let\veqno\@@leqno}

\makeatother

 \numberwithin{equation}  {section}
\begin{document}

\

\vspace{-2cm}

\title[A note on Voiculescu's theorem]{A note on the Voiculescu's theorem for commutative C$^*$-algebras in semifinite von Neumann algebras.}

\author{Don Hadwin}
\curraddr{Department of Mathematics \& Statistics, University of
New Hampshire, Durham, 03824, US}
\email{don@math.unh.edu}
\thanks{}
\author{Rui Shi}
\curraddr{School of Mathematical Sciences, Dalian University of
Technology, Dalian, 116024, P. R. China}
\email{ruishi@dlut.edu.cn, ruishi.math@gmail.com}
\thanks{Rui Shi was partly supported by NSFC(Grant No.11401071) and the Fundamental Research Funds for the Central Universities (Grant No.DUT16RC(4)57).}




\begin{abstract}
In the current paper, we generalize the ``compact operator'' part of D. Voiculescu's non-commmutative Weyl-von Neumann theorem on approximate equivalence of unital $*$-homomorphisms of an commutative C$^*$ algebra $\mathcal{A}$ into a semifinite von Neumann algebra. A result of D. Hadwin for approximate summands of representations into a finite von Neumann factor $\mathcal{R}$ is also extended.  
\end{abstract}

\maketitle

\section{Introduction}
In 1976, as a non-commutative version of the Weyl-von Neumann theorem \cite{Berg,Von2,Weyl}, Voiculescu \cite{Voi2} characterized approximate equivalence of two unital representations $\phi,\psi:\mathcal{A}\rightarrow \mathcal{B}(\mathcal{H})$, where $\mathcal{A}$ is a separable unital C$^*$-algebra and $\mathcal {H}$ is a complex separable Hilbert space. A different beautiful proof was given by Arveson \cite{Ave} in 1977. Two representations $\phi$ and $\psi$ of a C$^*$-algebra $\mathcal{A}$ on a Hilbert space $\mathcal{H}$ are said to be \emph{approximately $($unitarily$)$ equivalent}, denoted by $\phi \sim_{a}\psi$, if there exists a net $\left \{U_{\lambda}\right \}_{\lambda\in\Lambda}$ of unitary operators in $\mathcal {B}(\mathcal {H})$, the set of all the bounded linear operators, such that, for every $A\in \mathcal{A}$,%
\begin{equation}
	\lim_{\lambda}\left \Vert U_{\lambda}^{\ast}\phi \left(A\right)  U_{\lambda}-\psi \left(A\right)  \right \Vert =0. \label{equ1-1}
\end{equation}
When $\mathcal{A}$ is separable, $\{  U_{\lambda}\}_{\lambda\in\Lambda}$ can be chosen to be a sequence. Let $\mathcal{K}(\mathcal{H})$ denote the set of the compact operators on $\mathcal {H}$. We say that two representations $\phi$ and $\psi$ of a separable C$^*$-algebra $\mathcal{A}$ into $\mathcal {B}(\mathcal {H})$ are approximately unitarily equivalent (relative to $\mathcal{K}(\mathcal{H})$), denoted by $\phi\sim_{\mathcal{A}}\psi,\mod \mathcal{K}(\mathcal{H})$, if there exists the sequence $\left\{U_{n}\right \}^{\infty}_{n=1}$ of unitary operators in $\mathcal {B}(\mathcal {H})$ satisfying (\ref{equ1-1}) and%
\[
U_{n}^{\ast}\phi \left(A\right)U_{n}-\psi\left(A\right)\in
\mathcal{K}(\mathcal{H})
\]
for every $n$ and every $A\in \mathcal{A}$. If $\mathcal{A}$ is a non-unital C$^*$-algebra and $\sigma:\mathcal{A}\rightarrow \mathcal{B}\left(\mathcal {H}\right)$ is a $\ast$-homomorphism, then let $\mathcal {H}_1=\cap \left \{  \ker \sigma \left(  A\right):A\in\mathcal{A}\right\}$. We have%
\[
\sigma=\textbf{0}\oplus \sigma_{1}
\]
relative to the direct sum $\mathcal {H}=\mathcal {H}_1\oplus \mathcal {H}^{\perp}_1$. Thus $\sigma_{1}$ is said to be the
\emph{nonzero part} of $\sigma$.

The following is the theorem that Voiculescu proved in \cite{Voi2}.

\begin{theorem}
Suppose $\mathcal{A}$ is a separable unital C*-algebra, $\mathcal{H}$ is a separable
Hilbert space and $\phi,\psi:\mathcal{A}\rightarrow \mathcal{B}\left(\mathcal{H}\right)  $ are
unital $\ast$-homomorphisms. The following are equivalent:

\begin{enumerate}
\item $\phi \sim_{a}\psi.$

\item $\phi \sim_{\mathcal{A}}\psi \mod \mathcal{K}(\mathcal{H})$.

\item $\ker \phi=\ker \psi$, $\phi^{-1}\left(  \mathcal{K}(\mathcal{H})
\right)  =\psi^{-1}\left(  \mathcal{K}(\mathcal{H})\right)$, and the
nonzero parts of the restrictions $\phi|_{\phi^{-1}\left(\mathcal{K}(\mathcal{H})\right)}$ and $\psi|_{\psi^{-1}\left(\mathcal{K}(\mathcal{H})\right)}$ are unitarily equivalent.
\end{enumerate}
\end{theorem}

\bigskip

In \cite{Hadwin1}, the first author gave a different characterization of approximate equivalence. For $T\in \mathcal{B}(\mathcal{H})$, we let $\mathrm{rank}\left(T\right)$ denote the Hilbert-space dimension of the closure of the range \textrm{Ran}$\left(T\right)$ of $T$.

In the same paper, the first author (Lemma $2.3$ of \cite{Hadwin1})  proved an analogue for \emph{approximate summands. }

\begin{theorem}\label{approx-sum}
Suppose $\mathcal{A}$ is a separable unital C*-algebra, $\mathcal{H}$ and $\mathcal{K}$ are
Hilbert spaces, and $\phi:\mathcal{A\rightarrow}\mathcal{B}(\mathcal{H})$,
$\psi:\mathcal{A\rightarrow}\mathcal{B}(\mathcal{K})$ are unital representations.
The following are equivalent:

\begin{enumerate}
\item There is a representation $\gamma:\mathcal{A}\rightarrow \mathcal{B}(\mathcal{K}_1)$ for
some Hilbert space $\mathcal{K}_{1}$ such that%
\[
\psi \oplus \gamma \sim_{a}\phi.
\]

\item For every $A\in \mathcal{A}$,%
\[
\mathrm{rank}\left(\psi\left(A\right)\right)\leq\mathrm{rank}\left(
\phi \left(A\right)\right)  .
\]

\end{enumerate}
\end{theorem}

In her 1994 doctoral dissertation (see also \cite{Don1}), Huiru Ding extended some of these results to the case in which $\mathcal{B}(\mathcal{H})$ is replaced by a von Neumann algebra. The following are some terms adopted in this paper.

Suppose $\mathcal{R}$ is a von Neumann algebra and
$T\in \mathcal{R}$. We define the $\mathcal{R}$-\textrm{rank} of $T$ (denoted
by $\mathcal{R}$-\textrm{rank}$\left(T\right)$) to be the \emph{Murray-von
Neumann equivalence class} of the projection onto the closure of
\textrm{ran}$\left(T\right)$. If $\mathcal{A}$ is a unital C*-algebra and
$\phi, \psi:\mathcal{A}\rightarrow \mathcal{R}$ are unital $\ast$-homomorphisms,
we say that $\phi$ and $\psi$ are \emph{approximately equivalent in} $\mathcal{R}$, denoted by $\phi \sim_{a}\psi$ $\left(  \mathcal{R}\right)$, if there is a net $\{U_{\lambda}\}_{\lambda\in\Lambda}$ of unitary operators
\textbf{in} $\mathcal{R}$ such that, for every $A\in \mathcal{A}$,%
\[
\lim_{\lambda}\left \Vert U_{\lambda}^{\ast}\phi \left(A\right)  U_{\lambda
}-\psi \left(A\right)  \right \Vert =0.
\]

\begin{theorem}\label{DH}
Suppose $\mathcal{A}$ is a unital C*-algebra that is a direct limit of finite
direct sums of commutative C*-algebras tensored with matrix algebras, and if
$\mathcal{R}$ is a von Neumann algebra acting on a separable Hilbert space,
then the following are equivalent:

\begin{enumerate}
\item $\phi \sim_{a}\psi$ $\left(  \mathcal{R}\right).$

\item For every $A\in \mathcal{A}$,%
\[
\mathcal{R}\text{\textrm{-}}\mathrm{rank}\left(\phi\left(A\right)\right)
=\mathcal{R}\text{\textrm{-}}\mathrm{rank}\left(\psi\left(A\right)\right).
\]

\end{enumerate}
\end{theorem}

In the setting of von Neumann algebras. The compact ideal $\mathcal{K}(\mathcal{H})$ of $\mathcal{B}(\mathcal{H})$ can be extended in the following way.

In the current paper, we let $\mathcal{R}$ be a countably
decomposable, properly infinite von Neumann algebra with a faithful
normal semifinite tracial weight $\tau $. Let
\begin{equation}\label{compact-ideal-R-tau}
\begin{aligned}
  \mathcal P\mathcal F(\mathcal R,\tau) &=\{ P  \ : \ P=P^*=P^2\in \mathcal R \text { and } \tau(P)<\infty\},\\
  \mathcal F(\mathcal R,\tau) &= \{XPY \ : \ P\in  \mathcal P\mathcal F(\mathcal R,\tau)  \text { and } X,Y\in\mathcal R\},\\
  \mathcal K(\mathcal R,\tau)&= \text{$\|\cdot\|$-norm closure of }  \mathcal F(\mathcal R,\tau) \text{ in } \mathcal R,
\end{aligned}
\end{equation}%
be the sets of finite rank projections, finite rank operators, and compact
operators respectively, in $(\mathcal{R},\tau )$.

For a von Neumann algebra $\mathcal{R}$, denoted by $\mathcal{K}(\mathcal{R})$ the $\| \cdot\|$-norm
closed ideal generated by finite projections in $\mathcal{R}$. In general, $\mathcal{K}(\mathcal{R},\tau)$ is a subset of $\mathcal{K}(\mathcal{R})$. That is because a finite projection might
not be a finite rank projection with respect to $\tau$. However, if $\mathcal{R}$ is a countably decomposable semifinite factor, then it is true that
\[
\mathcal{K}(\mathcal{R},\tau) = \mathcal{K}(\mathcal{R})
\]
for a faithful, normal,
semifinite tracial weight $\tau$.

To extend the definition of approximate equivalence of two unital $\ast$-homomorphisms of a separable C$^{\ast}$-algebra $\mathcal{A}$ into $\mathcal{R}$ (\textbf{relative to $\mathcal{K}(\mathcal{R},\tau)$}), we need to develop the following notation and definitions.

Let $\mathcal{H}$ be an infinite dimensional separable Hilbert space and let $\mathcal{B}(\mathcal{H})$ be the set of bounded linear operators on $\mathcal{H}$. Suppose that $\{E_{i,j}\}_{i,j= 1}^\infty$ is a system of
matrix units of $\mathcal{B}(\mathcal{H})$.

For a countably
decomposable, properly infinite von Neumann algebra $\mathcal{R}$ with a faithful
normal semifinite tracial weight $\tau $.
There exists a sequence $\{V_i\}_{i = 1}^\infty$ of partial isometries in $\mathcal{R}$ such that
\begin{equation*}
V_iV_i^*=I_{\mathcal{R}}, \ \ \ \ \sum_{i = 1}^\infty V_i^*V_i=I_{\mathcal{R}}, \ \ \ \ \text{ and } V_jV_i^*=0 \text { when } i\ne j.
\end{equation*}

Let $\mathcal{R}\otimes \mathcal{B}(\mathcal{H})$ be a von Neumann algebra
tensor product of $\mathcal{R}$ and $\mathcal{B}(\mathcal{H})$.

\begin{definition}
For all $X\in \mathcal{R }$ and all \ $\sum_{i,j= 1}^\infty
X_{i,j}\otimes E_{i,j}\in \mathcal{R}\otimes \mathcal{B}(\mathcal{H})$,
define
\begin{equation*}
\phi: \mathcal{R}\rightarrow \mathcal{R}\otimes \mathcal{B}(\mathcal{H}) \ \
\ \text{ and } \ \ \psi: \mathcal{R}\otimes \mathcal{B}(\mathcal{H})
\rightarrow \mathcal{R}
\end{equation*}
by
\begin{equation*}
\phi(X) =\sum_{i,j= 1}^\infty (V_iXV^*_j)\otimes E_{i,j} \ \ \ \ \text{ and }
\ \ \ \ \psi( \sum_{i,j= 1}^\infty X_{i,j}\otimes E_{i,j} )= \sum_{i,j=1}^\infty V_i^*X_{i,j}V_j.
\end{equation*}
\end{definition}
By Lemma 2.2.2 of \cite{Li}, both $\phi$ and $\psi$ are normal $*$-homomorphisms
satisfying
\begin{equation*}
\text{ $\psi \circ \phi=id_{\mathcal{R}}$ \qquad and \qquad $\phi \circ
\psi=id_{\mathcal{R}\otimes \mathcal{B}(\mathcal{H})}.$}
\end{equation*}

\begin{definition}\label{weight++}
Define a mapping $\tilde \tau :(\mathcal{R}\otimes \mathcal{B}(\mathcal{H}))^+\rightarrow
[0,\infty]$ to be
\begin{equation*}
\tilde \tau (y)=\tau(\psi(y)), \qquad \forall \ y\in
(\mathcal{R}\otimes \mathcal{B}(\mathcal{H}))^+.
\end{equation*}
\end{definition}
By the above Definition, the following are proved in Lemma 2.2.4 of \cite{Li}:
\begin{enumerate}
\item[(i)] $\displaystyle\tilde \tau$ is a faithful, normal, semifinite tracial weight
of $\mathcal{R}\otimes \mathcal{B}(\mathcal{H})$.

\item[(ii)] $\displaystyle\tilde \tau(\sum_{i,j=1}^\infty X_{i,j}\otimes
E_{i,j})=\sum_{i=1}^\infty \tau(X_{i,i})$ for all \ $\displaystyle\sum_{i,j=1}^\infty X_{i,j}\otimes E_{i,j}\in (\mathcal{R}\otimes \mathcal{B}(\mathcal{H}))^+$.

\item[(iii)]
\begin{equation*}
\begin{aligned}
\mathcal P\mathcal F(\mathcal R\otimes \mathcal B(\mathcal H), \tilde \tau)&=\phi(\mathcal P\mathcal F(\mathcal R,\tau)), \quad \\
\mathcal F(\mathcal R\otimes \mathcal B(\mathcal H), \tilde \tau)&=\phi(\mathcal F(\mathcal R,\tau)), \quad \\
\mathcal K(\mathcal R\otimes \mathcal B(\mathcal H), \tilde \tau)&=\phi(\mathcal K(\mathcal R,\tau)).
\end{aligned}
\end{equation*}
\end{enumerate}

\begin{remark}
It shows that $\tilde \tau$ is a natural extension of $\tau$
from $\mathcal{R}$ to $\mathcal{R}\otimes \mathcal{B}(\mathcal{H})$. If no
confusion arises, $\tilde\tau$ will be also denoted by $\tau$. By Proposition $2.2.9$ of \cite{Li}, the ideal $\mathcal K(\mathcal R\otimes \mathcal B(\mathcal H), \tilde \tau)$ is independent of the choice of the
system of matrix units $\{E_{i,j}\}_{i,j=1}^\infty$  of $\mathcal{B}(\mathcal{H})$ and the choice of the family $\{V_i\}_{i =1}^\infty$ of
partial isometries in $\mathcal{R}$.
\end{remark}

Now we are ready to introduce the definition of approximate equivalence of $\ast$-homomorphisms of a separable C$^{\ast}$-algebra into $\mathcal{R}$ relative to $\mathcal{K}_{}(\mathcal{R},\tau)$.

Let $\mathcal{A}$ be a separable C$^*$-subalgebra of $\mathcal{R}$ with an identity $I_{\mathcal{A}}$.  Suppose that $\psi$ is a positive mapping from $\mathcal{A}$ into $\mathcal{R}$ such that $\psi(I_{\mathcal{A}})$ is a projection in $\mathcal{R}$. Then for all $0\le X\in \mathcal{A}$, we have $0\le \psi(X) \le \|X\|\psi(I_{\mathcal{A}})$. Therefore, it follows that
\[
\psi(X)\psi(I_{\mathcal{A}})=\psi(I_{\mathcal{A}})\psi(X)=\psi(X)
\]
for all positive $X\in \mathcal{A}$. In other words, $\psi(I_{\mathcal{A}})$ can be viewed as an identity of $\psi(\mathcal{A})$.
Or, $\psi(\mathcal{A})\subseteq \psi(I_{\mathcal{A}})\mathcal{R}\psi(I_{\mathcal{A}})$.

\begin{definition}$($Definition 2.3.1 of \cite{Li}$)$\label{AEC}
Suppose $\{E_{i,j}\}_{i,j\ge 1}$ is a system of
matrix units of $\mathcal{B}(\mathcal{H})$. Let $M,N\in \mathbb{N}\cup\{\infty\}$. Suppose that $\psi_1,\ldots, \psi_M$ and $\phi_1,\ldots, \phi_N$ are positive mappings from $\mathcal{A}$ into $\mathcal{R}$ such
that $\psi_1(I_{\mathcal{A}}),\ldots, \psi_M(I_{\mathcal{A}})$, $\phi_1(I_{\mathcal{A}}),\ldots, \phi_N(I_{\mathcal{A}})$ are projections in $\mathcal{R}$.

\begin{enumerate}
\item[(a)] Let $\mathcal{F}\subseteq \mathcal{A}$ be a finite subset and $\epsilon>0$. Say
\begin{equation*}
\text{ $\psi_1\oplus \cdots \oplus\psi_M$ is \emph{$(\mathcal{F},\epsilon)$-strongly-approximately-unitarily-equivalent} to $\phi_1\oplus\cdots\oplus\phi_N$ over $\mathcal{A}$,}
\end{equation*}
denoted by
\begin{equation*}
\psi_1\oplus\psi_2\oplus\cdots \oplus\psi_M \sim_{\mathcal{A}%
}^{(\mathcal{F},\epsilon)} \phi_1\oplus \phi_2\oplus\cdots\oplus \phi_N, \qquad \mod \mathcal{K}_{}(\mathcal{R},\tau) 
\end{equation*}
if there exists a partial isometry $V$ in $\mathcal{R}\otimes \mathcal{B}(\mathcal{H})$ such that

\begin{enumerate}
\item[(i)] $\displaystyle V^*V= \sum_{i=1}^M \psi_i(I_{\mathcal{A}})\otimes E_{i,i}$ and $\displaystyle VV^*= \sum_{i=1}^N \phi_i(I_{\mathcal{A}})\otimes E_{i,i}$;

\item[(ii)] $\displaystyle  \sum_{i=1}^M \psi_i(X)\otimes E_{i,i} - V^*\left(\sum_{i=1}^N \phi_i(X)\otimes E_{i,i}\right ) V \in \mathcal{K}(\mathcal{R}\otimes \mathcal{B}(\mathcal{H}),\tau)$ for all $X\in \mathcal{A}$;

\item[(iii)] $\displaystyle \|\sum_{i=1}^M \psi_i(X)\otimes
E_{i,i} - V^*\left (\sum_{i=1}^N \phi_i(X)\otimes E_{i,i}\right ) V
\|<\epsilon$ for all $X\in \mathcal{F}$.
\end{enumerate}

\item[(b)] Say
\begin{equation*}
\text{ $\psi_1\oplus \cdots \oplus\psi_M$ is \emph{strongly-approximately-unitarily-equivalent } to $\phi_1\oplus\cdots\oplus\phi_N$ over $\mathcal{A}$,}
\end{equation*}
denoted by
\begin{equation*}
\psi_1\oplus\psi_2\oplus\cdots \oplus\psi_M \sim_{\mathcal{A}} \phi_1\oplus
\phi_2\oplus\cdots\oplus \phi_N, \qquad \mod \mathcal{K}_{}(\mathcal{R},\tau)
\end{equation*}
if, for any finite subset $\mathcal{F}\subseteq \mathcal{A}$ and $\epsilon>0$,
\begin{equation*}
\psi_1\oplus\psi_2\oplus\cdots \oplus\psi_M \sim_{\mathcal{A}}^{(\mathcal{F},\epsilon)} \phi_1\oplus \phi_2\oplus\cdots\oplus \phi_N, \qquad \mod \mathcal{K}_{}(\mathcal{R},\tau).
\end{equation*}
\end{enumerate}
\end{definition}

In this paper we address the question of approximate summands and ``compact'' operators for semifinite von Neumann algebras $\mathcal{R}$ and commutative separable C*-algebras $\mathcal{A}$.  In Section 2, relative to finite von Neumann algebras, we characterize the approximate summands of $\ast$-homomorphisms by virtue of a natural condition. Precisely, we prove the following theorem.

\vspace{0.2cm} {T{\Small HEOREM}} \ref{Thm-2.2}
\emph{Suppose $\mathcal{A}$ is a separable unital commutative C*-algebra and
$\mathcal{R}$ is a finite von Neumann algebra acting on a separable Hilbert
space $\mathcal{H}$. Suppose $P$ is a projection in $\mathcal{R}$, $\pi:\mathcal{A}\rightarrow \mathcal{R}$ is a unital $\ast$-homomorphism and $\rho:\mathcal{A}\rightarrow P\mathcal{R}P$ is a unital $\ast$-homomorphism such that, for every $X\in \mathcal{A}$, we have%
\[
\mathcal{R}\text{\textrm{-}}\mathrm{rank}\left(  \rho \left(  X\right)
\right)\leq \mathcal{R}\text{\textrm{-}}\mathrm{rank}\left(\pi\left(
X\right)\right).
\]
Then there is a unital $\ast$-homomorphism $\gamma:\mathcal{A}\rightarrow P^{\perp}\mathcal{R}P^{\perp}$ such that
\[
\gamma \oplus \rho \sim_{a}\pi \text{ }\left(  \mathcal{R}\right).
\]}

In Section 3, for two $\ast$-homomorphisms $\phi$ and $\psi$ of a commutative $C^{\ast}$-algebra into a semifinite von Neumann factor $\mathcal{R}$ with a faithful normal semifinite tracial weight $\tau$, the main theorem states that the approximately unitary equivalence of $\phi$ and $\psi$ implies that these two $\ast$-homomorphisms are strongly-approximately-unitarily-equivalent over $\mathcal{A}$ (as in Definition \ref{AEC}). Precisely, we obtian the following theorem. 

\vspace{0.2cm} {T{\Small HEOREM}} \ref{App-II-infty}
\emph{Let $X$ be a compact metric space. Suppose that $\phi$ and $\psi$ are two unital $\ast$-homomorphisms of $C(X)$ into a countably decomposable,  properly
infinite,  semifinite factor $\mathcal{R}$ with a faithful normal semifinite tracial weight $\tau$ acting on a separable Hilbert space $\mathcal{H}$. Then the following are equivalent:
\begin{enumerate}
\item $\phi \sim_{a}\psi \, \left(\mathcal{R}\right)$,
\item $\phi\sim_{C(X)} \psi, \mod \mathcal{K}_{}(\mathcal{R},\tau)$.
\end{enumerate}}

\section{Representations relative to finite von Neumann algebras}

\begin{theorem}\label{2.1}
Suppose $\mathcal{A}$ is a separable unital commutative C*-algebra and
$\mathcal{R}$ is a type II$_1$ factor with a faithful normal normalized trace $\tau$ acting on a separable Hilbert
space $\mathcal{H}$. Suppose $P$ is a projection in $\mathcal{R}$, $\pi:\mathcal{A}\rightarrow \mathcal{R}$ is a unital $\ast$-homomorphism and $\rho:\mathcal{A}\rightarrow P\mathcal{R}P$ is a unital $\ast$-homomorphism such that, for every $X\in \mathcal{A}$, we have%
\[
\mathcal{R}\text{\textrm{-}}\mathrm{rank}\left(  \rho \left(  X\right)
\right)\leq \mathcal{R}\text{\textrm{-}}\mathrm{rank}\left(\pi\left(
X\right)\right).
\]
Then there is a unital $\ast$-homomorphism $\gamma:\mathcal{A}\rightarrow P^{\perp}\mathcal{R}P^{\perp}$ such that
\[
\gamma \oplus \rho \sim_{a}\pi \text{ }\left(  \mathcal{R}\right).
\]

\end{theorem}

\begin{proof}
It follows from Lemma $2.2$ of \cite{Takesaki} that $\pi$ and $\rho$ can be extended to normal unital $\ast$-homomorphisms with domain, the second dual $\mathcal{A}^{\# \#}$ of
$\mathcal{A}$, so that
\[
\mathcal{R}\text{\textrm{-}}\mathrm{rank}(\rho (X))\leq\mathcal{R}\text{\textrm{-}}\mathrm{rank}(\pi(X))
\]
holds for all $X\in \mathcal{A}^{\# \#}$. Since $\mathcal{A}$ is separable, we
can choose a countable family $\left \{  Q_{1},Q_{2},\ldots \right \}  $ of
projections in $\mathcal{A}^{\# \#}$ such that%
\[
\mathcal{A}\subseteq C^{\ast}\left(  Q_{1},Q_{2},\ldots \right)  .
\]
However, if we let $A=\sum_{k=1}^{\infty}{3^{-k}}Q_{k}$, then $C^{\ast}\left(A\right)  =C^{\ast}\left(Q_{1},Q_{2},\ldots \right)$. It is also true that, for every $X\in C^{\ast}\left(A\right)$,%
\[
\mathcal{R}\text{\textrm{-}}\mathrm{rank}\left(  \rho \left(  X\right)\right)  \leq \mathcal{R}\text{\textrm{-}}\mathrm{rank}\left(  \pi \left(X\right)  \right)  .
\]
It is easily seen that if we prove the theorem for the restrictions of $\pi$
and $\rho$ to $C^{\ast}\left(A\right)$, we will have proved the theorem
for $\pi$ and $\rho$ on $\mathcal{A}$. Hence we can assume that $\mathcal{A}%
=C^{\ast}\left(A\right)  $ and $0\leq A\leq 1$. Let $S=\rho \left(A\right)
\in P\mathcal{R}P$ and $T=\pi \left(A\right)  \in \mathcal{R}$. Thus the following inequality
\[
\mathcal{R}\text{\textrm{-}}\mathrm{rank}\left(  f\left(  S\right)P\right)
\leq \mathcal{R}\text{\textrm{-}}\mathrm{rank}\left(  f\left(  T\right)\right)
\]
holds for every $f\in C\left(  \sigma \left(A\right)  \right)$. This leads to the inequality
\[
\tau \left(  f\left(  S\right)  P\right)  \leq \tau \left(  f\left(  T\right)
\right)
\]
for every $f\in C\left(  \sigma \left(  A\right)  \right)$. Thus the Riesz representation theorem implies that there exist two regular
Borel measures $\mu_{\rho}$ and $\mu_{\pi}$ on $\sigma (A)  $ such
that the inequality
\[
\tau \left(  f\left(  S\right)  P\right)  =\int_{\sigma \left(  A\right)  }%
fd\mu_{\rho}\leq \int_{\sigma \left(  A\right)  }fd\mu_{\pi}=\tau \left(
f\left(  T\right)  \right)
\]
holds for every $f\in C\left(  \sigma \left( A\right)  \right)$.
It follows from Lusin's theorem that the preceding line holds for every
bounded Borel measurable function $f:\sigma \left( A\right)  \rightarrow
\mathbb{C}$. Hence $\mu_{\rho}\leq \mu_{\pi}$ and, for every $z\in \sigma \left(
A\right)$, we have $\tau \left(  \chi_{\left \{  z\right \}  }\left(  S\right)
\right)  \leq \tau \left(  \chi_{\left \{  z\right \}  }\left(  T\right)  \right)$.

Since $\tau$ is faithful, the set $L_{S}$ of $z\in \sigma \left(  S\right)$
such that $\chi_{\left \{  z\right \}  }\left(  S\right)  \neq0$ is countable. Hence $\sum_{z\in L_{S}}z\chi_{\left \{  z\right \}  }\left(  S\right)  $ is a
direct summand of $S$ and $\sum_{z\in L_{S}}z\chi_{\left \{  z\right \}
}\left(  T\right)  $ is a summand of $T$.

Since, for each $z\in L_{S}$, the
projection $\chi_{\left \{  z\right \}  }\left(  S\right)$ is unitarily
equivalent to a subprojection of $\chi_{\left \{  z\right \}  }\left(  T\right)
$, without loss of generality,
$\sum_{z\in L_{S}}z\chi_{\left \{  z\right \}  }\left(
S\right)$ can be assumed to be a direct summand of $T$. Thus this summand can be removed from
both $S$ and $T$. Therefore, it can be assumed that $S$ has no eigenvalues.

By the same way, the set $L_{T}=\left \{z\in \sigma \left( T\right)  :\chi_{\left \{ z\right \}  }\left(  T\right)  \neq0\right \} $ is countable.
Hence $S\chi_{L_{T}}\left(  S\right)  =0$. Therefore, for every bounded nonnegative measurable function
$f:\mathbb{R}\rightarrow \mathbb{R}$, we have
\[
\tau \left(  f\left(  S\right)  P\right)  =\tau \left(  \left(  \chi
_{\mathbb{C}\backslash L_{T}}f\right)  \left(  S\right)  P\right)  \leq
\tau \left(  \chi_{\mathbb{C}\backslash L_{T}}\left(  T\right)  f\left(
T\right)  \right)  d\mu_{\pi}.
\]
This yields that $T$ can be replaced with $T\left(  1-\chi_{\mathbb{C}\backslash L_{T}}\left(  T\right)  \right)$ and $\mathcal{R}$ can be replaced with
\[
\left(1-\chi_{\mathbb{C}\backslash L_{T}}\left(  T\right)  \right)  \mathcal{R}\left(  1-\chi_{\mathbb{C}\backslash L_{T}}\left(  T\right)  \right).
\]
Hence we can assume that $\chi_{L_{T}}\left(  T\right)  =0$.

Similarly, since the equality
\[
f\left(  S\right)  =\left(  \chi_{\sigma \left(  S\right)  }f\right)  \left(
S\right)
\]
holds for every
bounded measurable function $f$, the operator $T$ can be replaced with $\chi_{\sigma \left(  S\right)  }\left(  T\right)T$. Hence we can assume that $\sigma \left(  S\right)  =\sigma \left(  T\right)
=\sigma \left( A\right)$. Thus $\mu_{\rho}\leq \mu_{\pi}$ are both non-atomic
measures with supports satisfying $\sigma \left(  S\right)  =\sigma \left(
T\right)  =\sigma \left( A\right)$.
Moreover, we have the equalities
\[
\mu_{\rho}\left(\sigma \left(  A\right)  \right)  =\tau \left(  P\right) \quad \mbox{ and }\quad \mu_{\pi}\left(  \sigma \left(  A\right)  \right)  =1.
\]
It follows that $\nu=$ $\mu_{\pi}-\mu_{\rho}$ is a nonatomic measure and $\nu \left(  \sigma \left(  A\right)\right)  =1-\tau(P)$. Thus there is a unital weak*-continuous
$\ast$-isomorphism $\Delta_{S}:L^{\infty}\left[  0,\tau(P)
\right]  \rightarrow L^{\infty}(\mu_{\rho})  $ such that for
every $f\in L^{\infty}[0,\tau(P)]$,
\[
\int_{\sigma \left(  A\right)  }\Delta_{S}\left(  f\right)  d\mu_{\rho}%
=\int_{0}^{\tau \left(  P\right)  }f\left(  x\right)  dx.
\]
Similarly, there is an isomorphism $\Delta_{\nu}:L^{\infty}\left[  \tau \left(
P\right)  ,1\right]  \rightarrow L^{\infty}\left(  \nu \right)  $ such that the equality
\[
\int_{\sigma \left(  A\right)  }\Delta_{\nu}\left(  f\right)  d\nu=\int_{\tau \left(  P\right)  }^{1}f\left(  x\right)  dx.
\]
holds for every $f\in L^{\infty}[\tau(P),1]$.

Moreover, we can choose a maximal chain $\mathcal{C}=\{Q_{t}:0\leq t\leq 1-\tau(P)\}$ of projections in $P^{\bot}\mathcal{R}P^{\bot}$ with $\tau(Q_{t})=t$
for $0\leq t\leq 1-\tau(P)$. Thus there exists a weak*-continuous unital $\ast$-homomorphism $\Delta_{1}:L^{\infty}[\tau(P),1]\rightarrow W^{\ast}\left(  \mathcal{C}\right)$
such that, for every $t\in[0,1-\tau(P)]$, we have $\Delta_{1}(\chi_{\left[ \tau( P),\tau( P)+t\right)}) =Q_{t}$, and such that,
for every $f\in L^{\infty}\left[  \tau(  P),1\right]  $ we have%
\[
\tau \left(  \Delta_{1}\left(  f\right)  \right)  =\int_{\tau \left(  P\right)}^{1}f\left(  x\right)  dx.
\]
Define $\Delta:C(  \sigma( A)  )  \rightarrow
P\mathcal{R}P+P^{\bot}\mathcal{R}P^{\bot}\subset \mathcal {R}$ by
\[
\Delta \left(  h\right)  =h\left(  S\right)  \oplus \left(  \Delta_{1}%
\circ \Delta_{\nu}^{-1}\right)  \left(  h\right)  .
\]
If $z(\lambda)  =\lambda$ is the identity map on $\sigma(A)$, then $\Delta(z)  =S\oplus B$ and
\begin{align*}
\tau(\Delta(h))&=\tau(h(S))+\tau(\Delta_{1}(\Delta_{\nu}^{-1}(h)))\\
&=\int_{\sigma(A)  }hd\mu_{\rho}+\int_{\tau(P)}^{1}\Delta_{\nu}^{-1}(  h)(  x) dx\\
&=\int_{\sigma(A)  }hd\mu_{\rho}+\int_{\sigma(A)}h d\nu=\int_{\sigma (A)  }hd\mu_{\pi}=\tau(  h( T)).
\end{align*}
Hence for every $h\in C(\sigma (A))$, we have
$\tau(  h(  S\oplus B)) =\tau (  h(T))$. Define a unital $\ast$-homomorphism $\gamma:C(\sigma(A))  \rightarrow P^{\bot}\mathcal{R}P^{\bot}$ by
\[
\gamma \left(  h\right)  =P^{\bot}h(  B).
\]
By Theorem \ref{DH}, the above equality yields that $\rho \oplus \gamma \sim_{a}\pi$. This completes the proof.
\end{proof}

\begin{theorem}\label{Thm-2.2}
Suppose $\mathcal{A}$ is a separable unital commutative C*-algebra and
$\mathcal{R}$ is a finite von Neumann algebra acting on a separable Hilbert
space $\mathcal{H}$. Suppose $P$ is a projection in $\mathcal{R}$, $\pi:\mathcal{A}\rightarrow \mathcal{R}$ is a unital $\ast$-homomorphism and $\rho:\mathcal{A}\rightarrow P\mathcal{R}P$ is a unital $\ast$-homomorphism such that, for every $X\in \mathcal{A}$, we have%
\[
\mathcal{R}\text{\textrm{-}}\mathrm{rank}\left(  \rho \left(  X\right)
\right)\leq \mathcal{R}\text{\textrm{-}}\mathrm{rank}\left(\pi\left(
X\right)\right).
\]
Then there is a unital $\ast$-homomorphism $\gamma:\mathcal{A}\rightarrow P^{\perp}\mathcal{R}P^{\perp}$ such that
\[
\gamma \oplus \rho \sim_{a}\pi \text{ }\left(  \mathcal{R}\right).
\]
\end{theorem}

\begin{proof}
First, we suppose $\mathcal{R}$ is a II$_{1}$ von Neumann algebra acting on a
separable Hilbert space $\mathcal {H}$. By applying the central decomposition technique of von Neumann algebras, we can then write
\begin{align}
\mathcal {H}=L^{2}\left(\mu,\ell^{2}\right)=\int_{\Omega}^{\oplus}\ell^{2}d\mu \left(\omega \right)
\mbox{ and } \mathcal{R}=\int_{\Omega}^{\oplus}\mathcal{R}_{\omega}d\mu \left(\omega\right),\notag
\end{align}
where $\left(\Omega,\mu\right)$ is a probability space and each $\mathcal{R}_{\omega}$ is a II$_{1}$ factor with
a unique trace $\tau_{\omega}$. Furthermore, a faithful normal tracial state
$\tau$ on $\mathcal{R}$ can be defined in the following form
\[
\tau \left(\int_{\Omega}^{\oplus}A\left(\omega \right)d\mu \left(\omega \right)\right)
=\int_{\Omega}\tau_{\omega}\left(A\left(\omega \right)\right)d\mu\left(\omega\right).
\]
Similarly, the projection $P\in\mathcal{R}$ can be written in the form
\[
P=\int_{\Omega}^{\oplus}P\left(\omega \right)  d\mu \left(\omega \right)
\]
where $  P\left(  \omega \right)$ is a projection in
$\mathcal{R}_{\omega}$ a.e. $\left(  \mu \right)$.
Thus $P\mathcal{R}P$ can be written in the form
\[
P\mathcal{R}P=\int_{\Omega}^{\oplus}P\left(\omega\right)\mathcal{R}%
_{\omega}P\left(  \omega \right)  d\mu \left(  \omega \right).
\]
By Theorem \ref{2.1}, we can assume that $\mathcal{A}%
=C^{\ast}\left(A\right)$ and $0\leq A\leq 1$.
Thus, for the identity map $z(\lambda)=\lambda$ on $\sigma(A)$, suppose that $\pi \left(  z\right)  =T$ and $\rho \left(  z\right)  =S\in P\mathcal{R}P$. Then
we can write%
\[
T=\int_{\Omega}^{\oplus}T\left(  \omega \right)  d\mu \left(  \omega \right)
\]
and
\[
S=PSP=\int_{\Omega}^{\oplus}S\left(\omega \right)  d\mu \left(\omega \right)
=\int_{\Omega}^{\oplus}P\left(  \omega \right)  S\left(\omega \right)  P\left(  \omega \right)  d\mu \left(  \omega \right).
\]
It follows that, for every $f\in C\left(  \sigma \left( A\right)  \right)  $,%
\[
\pi \left(  f\right)  =f\left(  T\right)  =\int_{\Omega}^{\oplus}f\left(
T\left(  \omega \right)  \right)  d\mu \left(  \omega \right)  =\int_{\Omega
}^{\oplus}\pi_{\omega}\left(  f\right)  d\mu \left(  \omega \right).
\]
If $f\in C\left(  \sigma \left( A\right)  \right)  $ and $Q_{f\left(
T\right)}$ is the projection onto the closure of the range of $f\left(
T\right)$, then
\[
Q_{f\left(  T\right)}=\int_{\Omega}^{\oplus}Q_{f\left(T\left(\omega \right)  \right)}d\mu \left(\omega \right).
\]
Similarly, if $Q_{f\left(  S\right) P}$ is the range projection of $f\left(
S\right)P$, then
\[
Q_{f\left(  S\right) P}=\int_{\Omega}^{\oplus}Q_{f\left(  S\left(  \omega \right)  \right)
P\left(  \omega \right)  }d\mu \left(  \omega \right).
\]
If $\mathcal{R}$\textrm{-}$\mathrm{rank}\left(  f\left(  S\right)  P\right))\leq \mathcal{R}$\textrm{-}$\mathrm{rank}\left(  f\left(  T\right)  \right))$, then $Q_{f\left(  S\right)  P}$ is Murray-von Neumann equivalent to a
subprojection of $Q_{f\left(  T\right)  }$. Hence, for every central
projection $D$, we have $DQ_{f\left(  S\right) P}$ is Murray-von Neumann
equivalent to a subprojection of $DQ_{f\left(  T\right)  }$. Thus for every
measurable subset $E\subset \Omega$,
\[
\tau \left(  \chi_{E}Q_{f\left(  S\right)  P}\right)  \leq \tau \left(  \chi
_{E}Q_{f\left(  T\right)  }\right)  ,
\]
which means that%
\[
\int_{E}\tau_{\omega}\left(Q_{f\left(  S\left(  \omega \right)  \right)P\left(  \omega \right)  }\right)  d\mu \left(  \omega \right)  \leq
\int_{E}\tau_{\omega}\left(  Q_{f\left(  T\left(  \omega \right)  \right)  }\right)d\mu \left(  \omega \right)  .
\]
This yields that
\[
\tau_{\omega}\left(  Q_{f\left(  S\left(  \omega \right)  \right)  P\left(
\omega \right)  }\right)  \leq \tau_{\omega}\left(  Q_{f\left(  T\left(
\omega \right)  \right)  }\right)  \text{ a.e. }\left(  \mu \right).
\]
Since $C(  \sigma(  A))$ is separable, we conclude
that, except for a subset of $\Omega$ of measure $0$, for all $f\in C(\sigma(  A))$,%
\[
\tau_{\omega}\left(  Q_{f\left(  S\left(  \omega \right)  \right) P\left(
\omega \right)  }\right)  \leq \tau_{\omega}\left(  Q_{f\left(  T\left(
\omega \right)  \right)  }\right).
\]
We can now use Theorem \ref{2.1} and measurably choose $B\left(
\omega \right)  =B\left(  \omega \right)  ^{\ast}\in P\left(  \omega \right)
^{\bot}\mathcal{R}_{\omega}P\left(  \omega \right)  ^{^{\bot}}$ and define
\[
\gamma_{\omega}:C\left( A\right)  \rightarrow P\left(  \omega \right)  ^{\bot
}\mathcal{R}_{\omega}P\left(  \omega \right)  ^{^{\bot}}%
\quad\mbox{ by }\quad
\gamma_{\omega}\left(  f\right)  =f\left(  B\left(  \omega \right)  \right)P\left(  \omega \right)  ^{\bot}%
\]
so that%
\[
\pi_{\omega}\sim_{a}\rho_{\omega}\oplus \gamma_{\omega}\text{ }\left(
\mathcal{R}_{\omega}\right)  .
\]
It easily follows that if we define $\gamma \left(  f\right)  =\int_{\Omega
}^{\oplus}\gamma_{\omega}\left(  f\right)  d\mu \left(  \omega \right)$, then
$\pi \sim_{a}\rho \oplus \gamma$ $\left(  \mathcal{R}\right)$. This completes the proof.
\end{proof}

\section{Representations relative to semifinite infinite von Neumann algebras}

As shown in the proof of Theorem \ref{2.1}, it is sufficient to replace a separable commutative C*-algebra with some certain $C(X)$ on a compact metric space $X$. Suppose that $\mathcal{R}$ is a countably decomposable,  properly infinite, semifinite von Neumann algebra with a faithful, normal, semifinite tracial weight $\tau$. In this section, we need the following notation introduced in (\ref{compact-ideal-R-tau}). Let
\begin{equation*}
\begin{aligned}
  \mathcal P\mathcal F(\mathcal R,\tau) &=\{ P  \ : \ P=P^*=P^2\in \mathcal R \text { and } \tau(P)<\infty\},\\
  \mathcal F(\mathcal R,\tau) &= \{XPY \ : \ P\in  \mathcal P\mathcal F(\mathcal R,\tau)  \text { and } X,Y\in\mathcal R\},\\
  \mathcal K(\mathcal R,\tau)&= \text{$\|\cdot\|$-norm closure of }  \mathcal F(\mathcal R,\tau) \text{ in } \mathcal R,
\end{aligned}
\end{equation*}%
be the sets of finite rank projections, finite rank operators, and compact operators respectively, in $(\mathcal{R},\tau )$. For an operator $T\in\mathcal{R}$, denote by $R(T)$ the range projection onto the closure of the range of $T$.

In the rest of this section, we assume that $\mathcal{R}$ is a countably decomposable,  properly
infinite, semifinite von Neumann factor with a faithful, normal, semifinite tracial weight $\tau$. The following two lemmas are useful in the sequel.

\begin{lemma}\label{Tool-0}
For an operator $A$ in $\mathcal{R}$, the following are equivalent:
\begin{enumerate}
\item $A$ is in $\mathcal{K}(\mathcal{R},\tau)$;
\item $|A|$ is in $\mathcal{K}(\mathcal{R},\tau)$;
\item for every $\epsilon>0$, $\tau(\chi_{\left[0,\epsilon\right)}(|A|))=\infty$ and $\tau(\chi_{\left[\epsilon,\infty\right)}(|A|))<\infty$;
\item for every $\epsilon>0$, $\tau(\chi_{\left[0,\epsilon\right]}(|A|))=\infty$ and $\tau(\chi_{\left(\epsilon,\infty\right)}(|A|))<\infty$.
\end{enumerate}
\end{lemma}

\begin{proof}
For an operator $A$ in $\mathcal{R}$, Let $A=V|A|$ be the polar decomposition of $A$. If $A$ is in $\mathcal{K}(\mathcal{R},\tau)$, then so is $|A|=V^{*}A$. On the other hand, if $|A|$ is in $\mathcal{K}(\mathcal{R},\tau)$, then so is $A=V|A|$. That $(2)\Leftrightarrow(3)$ is equivalent to $(2)\Leftrightarrow(4)$. Thus, we only need to prove $(2)\Leftrightarrow(3)$. Suppose that $|A|$ belongs to $\mathcal{K}(\mathcal{R},\tau)$ and $\pi$ is the canonical $*$-homomorphism of $\mathcal{R}$ onto $\mathcal{R}/\mathcal{K}(\mathcal{R},\tau)$. If $\tau(\chi_{\left[0,\epsilon\right)}(|A|))<\infty$, then $\pi(\chi_{\left[0,\epsilon\right)}(|A|))=\pi(\chi_{\left[0,\epsilon\right)}(|A|)|A|)=0$. It follows that
\[
\pi(|A|)=\pi(\chi_{\left[0,\epsilon\right)}(|A|)+\chi_{\left[\epsilon,\infty\right)}(|A|)|A|).
\]
Note that $\chi_{\left[0,\epsilon\right)}(|A|)+\chi_{\left[\epsilon,\infty\right)}(|A|)|A|$ is invertible in $\mathcal{R}$, so $\pi(|A|)$ is invertible in $\mathcal{R}/\mathcal{K}(\mathcal{R},\tau)$. This is a contradiction. By a similar method, if $|A|$ belongs to $\mathcal{K}(\mathcal{R},\tau)$, then
\[
\chi_{\left[\epsilon,\infty\right)}(|A|)=\chi_{\left[\epsilon,\infty\right)}(|A|)|A||A|^{-1}\chi_{\left[\epsilon,\infty\right)}(|A|)\in\mathcal{K}(\mathcal{R},\tau).
\]
If $\chi_{\left[\epsilon,\infty\right)}(|A|)\in\mathcal{K}(\mathcal{R},\tau)$, then, for every $n\in\mathbb{N}$, there exists a positive operator $A_{n}$ such that
\begin{enumerate}
\item[(i)] $\tau(R(A_{n}))<\infty$;
\item[(ii)] $0\leq A_{n}\leq\chi_{\left[\epsilon,\infty\right)}(|A|)$;
\item[(iii)] $\Vert\chi_{\left[\epsilon,\infty\right)}(|A|)-A_{n}\Vert<1/n$.
\end{enumerate}
It is easy to obtain that $0\leq A_{n}\leq R(A_{n})\leq\chi_{\left[\epsilon,\infty\right)}(|A|)$. Thus
\[
\Vert\chi_{\left[\epsilon,\infty\right)}(|A|)-R(A_{n})\Vert<1/n.
\]
A routine calculation shows that $\chi_{\left[\epsilon,\infty\right)}(|A|)$ is unitarily equivalent to $R(A_{n})$. Therefore, we obtain that $\tau(\chi_{\left[\epsilon,\infty\right)}(|A|))<\infty$ holds for every $\epsilon>0$. By the definition of $\mathcal{K}(\mathcal{R},\tau)$, we can prove $(3)\Rightarrow(2)$. This completes the proof.
\end{proof}

\begin{lemma}\label{Tool-1}
Let $X$ be a compact metric space. Suppose that $\phi$ and $\psi$ are two unital $\ast$-homomorphisms of $C(X)$ into a countably decomposable,  properly
infinite,  semifinite factor $\mathcal{R}$ with a faithful normal semifinite tracial weight $\tau$ acting on a separable Hilbert space $\mathcal{H}$. If $\phi \sim_{a}\psi \, \left(\mathcal{R}\right)$, then, for $f$ in $C(X)$,
\[
\phi(f)\in\mathcal{K}(\mathcal{R},\tau)\Leftrightarrow\psi(f)\in\mathcal{K}(\mathcal{R},\tau).
\]
\end{lemma}

\begin{proof}
First, we need to extend $\phi$ and $\psi$ to $\hat{\phi}$ and $\hat{\psi}$ as two normal unital $\ast$-homomorphisms of $C(X)^{\#\#}$ into $\mathcal{R}$, respectively. Given any open subset $\Delta$ of $X$, there exists a continuous function $f$ such that
\[
\begin{cases}
0<f(x)\leq 1, & \text{if }x\in\Delta,\\
f(x)=0, & \text{if }x\notin\Delta.
\end{cases}
\]
Thus, the increasing sequence $\{f^{1/n}\}_{n\in\mathbb{N}}$ converges pointwise to $\chi_{\Delta}$. Furthermore, if $\{f^{1/n}\}_{n\in\mathbb{N}}$ are viewed as elements in $C(X)^{\#\#}$, then $f^{1/n}$ converges to $\chi_{\Delta}$ in the weak$^\ast$ topology. Since $\phi(f^{1/n})=\phi(f)^{1/n}$ and $\{\phi(f)^{1/n}\}_{n\in\mathbb{N}}$ is a monotone increasing sequence of positive operators in $\mathcal{R}$ with the upper bound $I_{\mathcal{R}}$. By applying Lemma $5.1.5$ of \cite{Kadison1}, $\phi(f)^{1/n}$ converges to the projection $R(\phi(f))$ in the strong operator topology. Therefore, $\phi$ can be extended to a unital normal $\ast$-homomorphism $\hat{\phi}$ of $\mathcal{B}(X)$, the $\ast$-subalgebra of all the bounded Borel functions on $X$, into $\mathcal{R}$ unambiguously such that $\hat{\phi}(\chi_{\Delta})=R(\phi(f))$. For details, the reader is referred to Theorem $5.2.6$ and Theorem $5.2.8$ of \cite{Kadison1}.

By applying Lemma \ref{Tool-0}, it is sufficient to suppose that $\phi(f)$ is a positive element in $\mathcal{K}(\mathcal{R},\tau)$. Thus, for every $\epsilon>0$, we have $\tau(\chi_{\left(\epsilon,\infty\right)}(\phi(f)))<\infty$. Note that there exists a continuous function $h$ defined as
\[
h(x)=
\begin{cases}
\dfrac{x-\epsilon}{x},&\epsilon< x,\\
0,&0\leq x\leq\epsilon
\end{cases}
\]
such that
\[
\chi_{\left(\epsilon,\infty\right)}(\phi(f))=\chi_{\left(\epsilon,\infty\right)}(\hat{\phi}(f))=\hat{\phi}(\chi_{\left(\epsilon,\infty\right)}(f))
=\hat{\phi}(\chi_{\left(0,\infty\right)}(h\circ f))=R({\phi}(h\circ f)).
\]
The same equality holds for $\psi$ and $\hat{\psi}$. By applying Theorem \ref{DH}, the relation $\phi \sim_{a}\psi \, \left(\mathcal{R}\right)$ yields the following equality
\[
\tau(\chi_{\left(\epsilon,\infty\right)}(\phi(f)))=\tau(R({\phi}(h\circ f)))=\tau(R({\psi}(h\circ f)))=\tau(\chi_{\left(\epsilon,\infty\right)}(\psi(f))).
\]
A similar argument ensures that
\[
\tau(\chi_{\left[0,\epsilon\right)}(\phi(f)))=\tau(\chi_{\left[0,\epsilon\right)}(\psi(f))).
\]
Therefore, $\phi(f)$ in $\mathcal{K}(\mathcal{R},\tau)$ implies $\psi(f)$ in $\mathcal{K}(\mathcal{R},\tau)$, and vice versa.
\end{proof}

Suppose that $\phi$ and $\psi$ as assumed are two unital $\ast$-homomorphisms of $C(X)$ into $\mathcal{R}$. Then, by Definition \ref{AEC}, the relation $\phi\sim_{C(X)} \psi, \mod \mathcal{K}_{}(\mathcal{R},\tau)$ implies that $\phi \sim_{a}\psi \, \left(\mathcal{R}\right)$. In the rest of this section, we aim to prove the inverse of this.




\begin{theorem}\label{App-II-infty}
Let $X$ be a compact metric space. Suppose that $\phi$ and $\psi$ are two unital $\ast$-homomorphisms of $C(X)$ into a countably decomposable,  properly
infinite,  semifinite factor $\mathcal{R}$ with a faithful normal semifinite tracial weight $\tau$ acting on a separable Hilbert space $\mathcal{H}$. Then the following are equivalent:
\begin{enumerate}
\item $\phi \sim_{a}\psi \, \left(\mathcal{R}\right)$,
\item $\phi\sim_{C(X)} \psi, \mod \mathcal{K}_{}(\mathcal{R},\tau)$.
\end{enumerate}
\end{theorem}
\begin{proof}
Assume that $\phi$ and $\psi$ are approximately unitarily equivalent relative to $\mathcal{R}$. By applying Theorem \ref{DH}, for every $f$ in $C(X)$, the equality
\[
\mathcal{R}\text{-}\mbox{rank}(\phi(f))=\mathcal{R}\text{-}\mbox{rank}(\psi(f))
\]
holds and yields that $\tau(R(\phi(f)))=\tau(R(\psi(f)))$. Thus, the equality $\ker\phi=\ker\psi$ holds. This ensures that $\psi\circ\phi^{-1}$ is a well-defined unital $\ast$-isomorphism of $\phi(C(X))$ onto $\psi(C(X))$ and we denote this isomorphism by $\rho$. That is, for every $A$ in $\phi(C(X))$ and every $f$ in $C(X)$, 
\[
\rho(A)=\psi\circ\phi^{-1}(A),\quad\rho(\phi(f))=\psi(f).
\]
Therefore, the following two statements are equivalent
\begin{enumerate}
	\item $\phi\sim_{C(X)} \psi, \mod \mathcal{K}_{}(\mathcal{R},\tau)$;
	\item $\rm{id}\sim_{\phi(C(X))} \rho, \mod \mathcal{K}_{}(\mathcal{R},\tau)$, where $\rm{id}$ stands for the identity mapping.
\end{enumerate}

In the following, we need to partition $X$ into two parts in order to reduce the proof into two special cases. Then we assemble them to complete the proof.

By a routine computation, it is easy to verify that the set
\[
\mathcal{I}=\{f\in C(X):\phi(f)\in\mathcal{K}(\mathcal{R},\tau)\}
\]
is a closed ideal in $C(X)$.
Note that, by Lemma \ref{Tool-1}, the equality
\[
\phi(C(X))\cap\mathcal{K}(\mathcal{R},\tau)=\psi(C(X))\cap\mathcal{K}(\mathcal{R},\tau)
\]
holds. This implies that the equality $\mathcal{I}=\{f\in C(X):\psi(f)\in\mathcal{K}(\mathcal{R},\tau)\}$ also holds.

By applying Theorem $3.4.1$ of \cite{Kadison1}, there exists a closed subset $F$ of the compact metric space $X$ such that
\[
\mathcal{I}=\{f\in C(X):f(x)=0, \forall x\in F\}.
\]
As shown in Lemma \ref{Tool-1}, we denote by $\hat{\phi}$ and $\hat{\psi}$ the normal extensions of $\mathcal{B}(X)$ into $\mathcal{R}$ induced by $\phi$ and $\psi$, respectively.
Note that, for every $f$ in $C(X)$, the projections $\hat{\phi}(\chi_{F})$ and $\hat{\psi}(\chi_{F})$ reduce $\phi(f)$ and $\psi(f)$, respectively.

To deal with one of the two special cases mentioned above, we adopt the classical method initiated by Voiculescu. That is, for every $A\in{\phi}(C(X))$ and $B\in{\psi}(C(X))$, we can define representations $\rho_{e}$ and $\rho^{\prime}_{e}$ as follows
\[
\rho_{e}(A)\triangleq \psi\circ\phi^{-1}(A)|_{\text{ran}\ \hat{\psi}(\chi_{F})},\quad \rho^{\prime}_{e}(B)\triangleq \phi\circ\psi^{-1}(B)|_{\text{ran}\ \hat{\phi}(\chi_{F})}.
\]
Note that
\[
\rho_{e}({\phi}(C(X))\cap\mathcal{K}(\mathcal{R},\tau))=\rho^{\prime}_{e}({\psi}(C(X))\cap\mathcal{K}(\mathcal{R},\tau))=0.
\]
By applying Theorem $5.3.1$ of \cite{Li}, we have
\[
{\rm{id}}_{{\phi}(C(X))}\sim_{\phi(C(X))} {\rm{id}}_{{\phi}(C(X))}\oplus\rho_{e},\quad \mod \mathcal{K}_{}(\mathcal{R},\tau),
\]

\[
{\rm{id}}_{{\psi}(C(X))}\sim_{\psi(C(X))} {\rm{id}}_{{\psi}(C(X))}\oplus\rho^{\prime}_{e},\quad \mod \mathcal{K}_{}(\mathcal{R},\tau).
\]

Therefore, for every $f\in C(X)$, it follows that
\begin{equation}
\phi(f)\sim_{C(X)} \phi(f)\oplus(\psi(f)|_{\text{ran}\ \hat{\psi}(\chi_{F})}), \mod \mathcal{K}_{}(\mathcal{R},\tau)\label{equ3.1}
\end{equation}
and
\begin{equation}
\psi(f)\sim_{C(X)} \psi(f)\oplus(\phi(f)|_{\text{ran}\ \hat{\phi}(\chi_{F})}), \mod \mathcal{K}_{}(\mathcal{R},\tau). \label{equ3.2}
\end{equation}
Note that, for every $f\in C(X)$, the equalities
\begin{equation}
\phi(f)\oplus(\psi(f)|_{\text{ran}\ \hat{\psi}(\chi_{F})})
=(\phi(f)|_{\text{ran}\ \hat{\phi}(\chi_{(X-F)})})\oplus(\phi(f)|_{\text{ran}\ \hat{\phi}(\chi_{F})})\oplus(\psi(f)|_{\text{ran}\ \hat{\psi}(\chi_{F})})\label{equ3.3}
\end{equation}
and
\begin{equation}
\psi(f)\oplus(\phi(f)|_{\text{ran}\ \hat{\phi}(\chi_{F})})
=(\psi(f)|_{\text{ran}\ \hat{\psi}(\chi_{(X-F)})})\oplus(\psi(f)|_{\text{ran}\ \hat{\psi}(\chi_{F})})\oplus(\phi(f)|_{\text{ran}\ \hat{\phi}(\chi_{F})})\label{equ3.4}
\end{equation}
hold. Thus, the above relations from $(3.1)$ to $(3.4)$ imply that, to prove that
\[
\phi\sim_{C(X)} \psi, \mod \mathcal{K}_{}(\mathcal{R},\tau),
\]
it is sufficient to prove that
\[
\phi|_{\text{ran}\ \hat{\phi}(\chi_{(X-F)})}\sim_{C(X)}
\psi|_{\text{ran}\ \hat{\psi}(\chi_{(X-F)})}, \mod \mathcal{K}_{}(\mathcal{R},\tau).
\] And this is the other special case.

For every $f$ in $C(X)$, write
\begin{align*}
\phi_{0}(f)\triangleq\phi(f)|_{\text{ran}\ \hat{\phi}(\chi_{(X-F)})},\\
\psi_{0}(f)\triangleq\psi(f)|_{\text{ran}\ \hat{\psi}(\chi_{(X-F)})}.
\end{align*}

Since $X$ is a compact metric space and $F$ is a closed subset of $X$, we can construct a continuous function $h$ such that
\[
h(x)=\mbox{dist}(x,F),\quad\forall x\in X,
\]
where $\mbox{dist}(x,F)$ is the distance between $x$ and $F$. This construction of $h$ ensures that $\phi(h)$ is bounded and belongs to $\mathcal{K}(\mathcal{R},\tau)$.
By applying Lemma \ref{Tool-0}, it follows that:
\begin{enumerate}
\item for every positive integer $k$, the projection $\hat{\phi}(\chi_{\left(\frac{1}{k},\infty\right)}(h))$ is finite, i.e.,
\[
\tau(\hat{\phi}(\chi_{\left(\frac{1}{k},\infty\right)}(h)))<\infty;
\]
\item for every positive integer $k$,
\[
\hat{\phi}(\chi_{\left(\frac{1}{k},\infty\right)}(h))\le\hat{\phi}(\chi_{\left(\frac{1}{k+1},\infty\right)}(h));
\]
\item as $k$ goes to infinity, the projection $\hat{\phi}(\chi_{\left(\frac{1}{k},\infty\right)}(h))$ converges to $\hat{\phi}(\chi_{(X-F)})$ in the strong operator topology.
\end{enumerate}

For a fixed $\delta>0$, define a closed subset $\Delta$ of $X$ by 
\[
\Delta\triangleq\{x\in X: \mbox{dist}(x,F)=\delta\}.
\]
Then $\hat{\phi}({\chi}_{\Delta})$ is a sub-projection of certain $\hat{\phi}(\chi_{\left(\frac{1}{k},\infty\right)}(h))$. Therefore, there exist at most countably many such $\hat{\phi}({\chi}_{\Delta})$ satisfying $\tau(\hat{\phi}({\chi}_{\Delta}))>0$. This implies that there exists a decreasing sequence $\{\alpha_{k}\}^{\infty}_{k=1}$ in the unit interval converging to $0$ such that
\begin{equation}
\hat{\phi}(\chi_{\left(\alpha_{k+1},\alpha_{k}\right)}({h}))=\hat{\phi}(\chi_{\left(\alpha_{k+1},\alpha_{k}\right]}({h}))=\hat{\phi}(\chi_{\left[\alpha_{k+1},\alpha_{k}\right]}({h})).	\label{equ3.5}
\end{equation}
Write $\alpha_{0}=+\infty$. For every $k$ in $\mathbb{N}$,
\[
\tau(\hat{\phi}(\chi_{\left(\alpha_{k+1},\alpha_{k}\right]}({h})))<\infty.
\]

Note that, for every $k\geq 1$, $\Delta_{k}\triangleq\{x\in X:\alpha_{k}<\mbox{dist}(x,F)<\alpha_{k-1}\}$ is open in $X$. Thus, there exists a positive continuous function $h_{k}$ satisfying
\begin{enumerate}
	\item $0\leq h_{k}\leq 1,\quad \forall k\geq 1$;
	\item $h_{k}(x)>0,\,\,\forall x\in \Delta_{k}$;
	\item $h_{k}(x)=0,\,\,\forall x\in X\backslash\Delta_{k}$;
	\item $R({\phi}({h}_{k}))=\hat{\phi}(\chi_{\left(\alpha_{k},\alpha_{k-1}\right)}({h}))$.
\end{enumerate}
Since $\tau(R(\phi({h}_{k})))=\tau(R(\psi({h}_{k})))<\infty$,
the reduced von Neumann algebras
\[
\mathcal{N}_{k}=\hat{\phi}(\chi_{\left(\alpha_{k},\alpha_{k-1}\right]}({h}))\mathcal{R}\hat{\phi}(\chi_{\left(\alpha_{k},\alpha_{k-1}\right]}({h}))
\]
and
\[ \mathcal{M}_{k}=\hat{\psi}(\chi_{\left(\alpha_{k},\alpha_{k-1}\right]}({h}))\mathcal{R}\hat{\psi}(\chi_{\left(\alpha_{k},\alpha_{k-1}\right]}({h}))
\]
are both type $II_1$ factors.

Furthermore, for every $f$ in $C(X)$ and $k\geq 1$, define two $\ast$-homomorphisms $\phi_{k}$ and $\psi_{k}$ of $C(X)$ into $\mathcal{R}$ by
\[
{\phi_{k}}(f)=\hat{\phi}(\chi_{\left(\alpha_{k},\alpha_{k-1}\right]}({h})f),\quad
{\psi_{k}}(f)=\hat{\psi}(\chi_{\left(\alpha_{k},\alpha_{k-1}\right]}({h})f)
\]
belonging to $\mathcal{N}_{k}$ and $\mathcal{M}_{k}$, respectively.

Note that the equality
\begin{equation*}
\tau(R(\phi({h}_{k}f)))=\tau(R(\psi({h}_{k}f))) 
\end{equation*}
implies
\begin{equation}
	\tau(R(\phi_{k}(f)))=\tau(R(\hat{\phi}(\chi_{\left(\alpha_{k+1},\alpha_{k}\right]}({h})f)))=\tau(R(\hat{\psi}(\chi_{\left(\alpha_{k+1},\alpha_{k}\right]}({h})f)))=\tau(R(\psi_{k}(f))). \label{equ3.6}
\end{equation}
Therefore, by applying (\ref{equ3.6}) and Theorem \ref{DH}, for every $k\geq 1$, the relation
\begin{equation}
	{\phi}_{k}\sim_{C(X)} {\psi}_{k},\quad \mod \mathcal{K}_{}(\mathcal{R},\tau)\label{rel3.7}
\end{equation}
holds.

Since $X$ is a compact matric space, there exists a sequence $\mathcal{B}=\{f_{i}\}_{i\in\mathbb{N}}$ dense in $C(X)$. By applying (\ref{rel3.7}), there exists a sequence $\{V_{mk}\}^{\infty}_{m,k=1}$ of unitary operators from $\mathcal{M}_{k}$ to $\mathcal{N}_{k}$ such that
\[
\Vert V^{*}_{mk}{\phi}_{k}(f_{i})V^{}_{mk}-{\psi}_{k}(f_{i})\Vert<\frac{1}{2^m}\cdot\frac{1}{2^k},\quad 1\leq i\leq m+k.
\]
Define a partial isometry $V_{m}$ by $V_{m}\triangleq\oplus^{\infty}_{k=1}V_{mk}$. Then, it follows that
\begin{enumerate}
\item [(a)] for every $m\geq 1$,
\[
V^{*}_{m}V^{}_{m}=\psi_{0}(1)\quad \text{ and }\quad V^{}_{m}V^{*}_{m}=\phi_{0}(1);
\]
\item [(b)] for every $f$ in $C(X)$ and every $m\geq 1$ the limit
\[
\sum^{\infty}_{k=1}\Vert V^{*}_{mk}{\phi}_{k}(f)V^{}_{mk}-{\psi}_{k}(f)\Vert<\infty
\]
shows that $V^{*}_{m}{\phi}_{0}(f)V^{}_{m}-{\psi}_{0}(f)$ is in $\mathcal{K}_{}(\mathcal{R},\tau)$;
\item [(c)] for every $f$ in $C(X)$, there corresponds a sufficiently large $m$, such that
\[
\Vert V^{\ast}_{m}\phi_{0}(f)V_{m}-\psi_{0}(f)\Vert<\frac{1}{2^m}.
\]
\end{enumerate}
By the definition, the above (a), (b), and (c) lead to that
\[
\phi_{0}\sim_{C(X)}\psi_{0},\mod \mathcal{K}_{}(\mathcal{R},\tau).
\]
Thus, combining the above reductions, we obtain that
\[
\phi_{}\sim_{C(X)}\psi_{},\mod \mathcal{K}_{}(\mathcal{R}_{},\tau_{}).
\]
This completes the proof.
\end{proof}

\vspace{1cm}

\end{document}